\documentclass[fleqn]{amsart}
\usepackage{mathrsfs}

\newtheorem{theorem}{Theorem}[section]
\newtheorem{lemma}[theorem]{Lemma}
\newtheorem{prop}[theorem]{Proposition}
\newtheorem{cor}[theorem]{Corollary}
\theoremstyle{definition}

\theoremstyle{remark}

\numberwithin{equation}{section}

\begin{document}

\title{general solutions to equation $axb^*-bx^*a^*=c$ in rings with involution}

\author{Chao You}
\address{Department of Mathematics $\&$ The Academy of Fundamental and Interdisciplinary Science\\Harbin Institute of Technology\\
Harbin 150001, Heilongjiang, People's Republic of China}
\email{hityou1982@gmail.com}

\author{Changhui Wang}
\address{Department of Mathematics\\Harbin Institute of Technology\\
Harbin 150001, Heilongjiang, People's Republic of China}
\email{shuzihero@gmail.com}

\author{Yicheng Jiang}
\address{Department of Mathematics\\Qiqihar University\\Qiqihar, 161006, Heilongjiang, People's Republic of China}
\email{hitjyc@gmail.com}



\dedicatory{This paper is dedicated to Prof. Lixin Xuan.}


\begin{abstract}
In [Q. Xu et al., The solutions to some operator equations, Linear
Algebra Appl.(2008), doi:10.1016/j.laa.2008.05.034], Xu et al.
provided the necessary and sufficient conditions for the existence
of a solution to the equation $AXB^*-BX^*A^*=C$ in the general
setting of the adjointable operators between Hilbert $C^*$-modules.
Based on the generalized inverses, they also obtained the general
expression of the solution in the solvable case. In this paper, we
generalize their work in the more general setting of ring
$\mathscr{R}$ with involution $*$ and reobtain results for
rectangular matrices and operators between Hilbert $C^*$-modules by
embedding the ``rectangles'' into rings of square matrices or rings
of operators acting on the same space.
\end{abstract}

\maketitle

\section*{Introdution}
Let $R(A)$ be the range of a matrix or an operator. The equation
$AXB^*-BX^*A^*=C$ was studied by Yuan \cite{yuan} and Xu et al.
\cite{xu}, for finite matrices and adjointable operators between
Hilbert $C^*$-modules respectively, under the condition that
$R(B)\subseteq R(A)$. When $A$ equals an identity matrix or identity
operator, this equation reduces to $XB^*-BX^*=C$, which was studied
by Braden \cite{braden} for finite matrices, and Djordjevi\'{c}
\cite{djordjevic} for the Hilbert space operators.

In this paper, we turn our attention to the equations
$axb^*-bx^*a^*=c$ where $a,b,c$ and $x$ are elements of a ring
$\mathscr{R}$ with involution. This point of view emphasizes the
purely algebraic nature of the problem without regard to specific
properties of matrices or bounded linear operators, and reveals the
intrinsic simplicity of the solutions. Thus the equations are
studied in a greater generality and in a transparent environment.

A novel feature of our paper is that the results for finite
rectangular matrices and adjointable operators between Hilbert
$C^*$-modules are derived from theorems for rings with involution
using the method of embedding decribed in Section \ref{embedding}.

Having established preliminary settings in Section
\ref{preliminary}, we study the equation $AXB^*-BX^*A^*=C$ in
Section \ref{generalization} in the setting of rings with
involution, giving necessary and sufficient conditions for the
existence, and the general form of these solutions based on
Moore-Penrose inverses.

Section \ref{embedding} is concerned with the extensions of the
preceding results to finite rectangular matrices with entries in a
ring with involution, and to adjointable operators between Hilbert
$C^*$-modules. This is achieved by embedding the rectangular
matrices as ¡®blocks¡¯ into the ring of square matrices of the same
order, and by embedding ¡®rectangular¡¯ operators via operator
matrices into the ring of operators acting on the same space. This
shows that our work is genuinely a generalization of previous
results, and a progress in the theory of equations in this type.

\section{Preliminaries}\label{preliminary}
Throughout this paper, ring $\mathscr{R}$ will always mean an
involution ring with a unit $1\neq 0$ such that $2$ is invertible in
$\mathscr{R}$. An involution $*$ is a unary operation $a\mapsto a^*$
on $\mathscr{R}$ preserved by the addition ($(a+b)^*=a^*+b^*$),
reversed by the multiplication ($(ab)^*=b^*a^*$) and satisfying
$(a^*)^*=a$ and $1^*=1$.

If $\mathscr{R}$ is a ring with involution and $a \in \mathscr{R}$,
we say that $b \in \mathscr{R}$ is a \emph{Moore-Penrose inverse} of
$a$, or \emph{MP-inverse} for short, if it satisfies the
\emph{Penrose equations} \cite{penrose}:
\begin{equation}\label{penrose}
aba=a, \quad bab=b, \quad (ab)^*=ab, \quad (ba)^*=ba.
\end{equation}

\begin{prop}\label{uniqueness}
Let $\mathscr{R}$ be a ring with involution and $a \in \mathscr{R}$.
If $b_1,b_2\in \mathscr{R}$ are both MP-inverses of $a$, then
$b_1=b_2$.
\end{prop}

\begin{proof}
$b_1a=b_1ab_2a=a^*b_1^*a^*b_2^*=a^*b_2^*=b_2a$. Similarly,
$ab_1=ab_2$. Hence $b_1=b_1ab_1=b_2ab_1=b_2ab_2=b_2$.
\end{proof}

From Proposition \ref{uniqueness}, we know that the MP-inverse of
$a$ is unique if it exists, and is denoted by $a^{\dag}$. If the
MP-inverse $a^{\dag}$ of $a$ exists, we say that $a$ is
\emph{Moore-Penrose invertible}, or \emph{MP-invertible} for short.

\begin{prop}
$2^{-1}$, the inverse of $2$ in $\mathscr{R}$, commutes with every
element in $\mathscr{R}$.
\end{prop}

\begin{proof}
For any $r\in \mathscr{R}$, $2r=(1+1)r=r+r=r(1+1)=r2$. Multiplying
$2^{-1}$ from both left and right sides of the equality above, we
get $r2^{-1}=2^{-1}r$ as desired.
\end{proof}

Since, moreover, it can be easily checked that $2^{-1}+2^{-1}=1$ in
$\mathscr{R}$, we can see that $2^{-1}$ functions just as number
``$\frac{1}{2}$'' in the calculations within $\mathscr{R}$. Hence,
without any confusion, we will denote $2^{-1}$ by $\frac{1}{2}$
throughout this paper.

More notations are needed. In this paper, $E_a$, $F_a$,
$H^{(+,*)}(a)$ and $H^{(-,*)}(a)$ are reversed to denote
$1-aa^{\dag}$, $1-a^{\dag}a$, $a+a^*$ and $a-a^*$, respectively.
Please keep in mind that $E_a$, $F_a$ are projections, and
$(H^{(+,*)}(a))^*=H^{(+,*)}(a)$, $(H^{(-,*)}(a))^*=-H^{(-,*)}(a)$,
which will be very useful in the calculations later.

\section{General solutions to the equation $axb^*-bx^*a^*=c$ in
the setting of rings with involution}\label{generalization}

In this section, we will study the general solutions to
Eq.(\ref{main2}) below in the general setting of rings with
involution.

\begin{lemma}
If $a,b$ are MP-invertible elements in $\mathscr{R}$ such that
$aa^{\dag}b=b$ and $(a^{\dag}bb^{\dag}a)^*=a^{\dag}bb^{\dag}a$, then
$d=E_ba$ is also MP-invertible with the unique MP-inverse
$d^{\dag}=a^{\dag}E_b$.
\end{lemma}

\begin{proof}
Check the conditions of (\ref{penrose}):
\begin{align*}
(E_ba)(a^{\dag}E_b)(E_ba)&=E_baa^{\dag}E_ba=E_baa^{\dag}(1-bb^{\dag})a
=E_b(aa^{\dag}a-aa^{\dag}bb^{\dag}a)\\
&=E_b(a-bb^{\dag}a)=E_b(1-bb^{\dag})a=E_bE_ba=E_ba.
\end{align*}
Since $aa^{\dag}b=b$, then $aa^{\dag}bb^{\dag}=bb^{\dag}$, and
$bb^{\dag}aa^{\dag}=bb^{\dag}$.
\begin{align*}
(a^{\dag}E_b)(E_ba)(a^{\dag}E_b)&=a^{\dag}E_baa^{\dag}E_b
=a^{\dag}(1-bb^{\dag})aa^{\dag}E_b=(a^{\dag}aa^{\dag}-a^{\dag}bb^{\dag}aa^{\dag})E_b\\
&=(a^{\dag}-a^{\dag}bb^{\dag})E_b=a^{\dag}(1-bb^{\dag})E_b=a^{\dag}E_bE_b=a^{\dag}E_b.
\end{align*}
\begin{equation*}
((E_ba)(a^{\dag}E_b))^*=(E_b(aa^{\dag})E_b)^*=E_b(aa^{\dag})E_b=(E_ba)(a^{\dag}E_b).
\end{equation*}
\begin{equation*}
((a^{\dag}E_b)(E_ba))^*=(a^{\dag}E_ba)^*=(a^{\dag}a-a^{\dag}bb^{\dag}a)^*=a^{\dag}a-a^{\dag}bb^{\dag}a=(a^{\dag}E_b)(E_ba).
\end{equation*}
Thus $d$ is MP-invertible and $d^{\dag}=a^{\dag}E_b$.
\end{proof}

\begin{theorem}
Let $a,b$ be MP-invertible elements in $\mathscr{R}$ such that
$aa^{\dag}b=b$ and $(a^{\dag}bb^{\dag}a)^*=a^{\dag}bb^{\dag}a$, and
$d=E_ba$. Then the general solution $x\in \mathscr{R}$ to the
equation
\begin{equation}\label{main1}
axb^*-bx^*a^*=0
\end{equation}
is of the form
\begin{equation}
x=v-\frac{1}{2}a^{\dag}avb^{\dag}b+\frac{1}{2}a^{\dag}bv^*a^*(b^{\dag})^*
-\frac{1}{2}a^{\dag}bv^*(b^{\dag}ad^{\dag}a)^*-\frac{1}{2}d^{\dag}avb^{\dag}b,
\end{equation}
where $v\in \mathscr{R}$ is arbitrary.
\end{theorem}

\begin{proof}
Since $d=E_ba$ by definition, we have $E_bd=d$, hence
$E_bdd^{\dag}=dd^{\dag}$. Taking $*$-operation, we get
$dd^{\dag}=dd^{\dag}E_b$. It follows that

\begin{align}
&d^{\dag}b=d^{\dag}(dd^{\dag}b)=d^{\dag}(dd^{\dag}E_bb)=0,\quad
b^*dd^{\dag}=(dd^{\dag}b)^*=0, \label{2.5}\\
&dd^{\dag}a=dd^{\dag}(bb^{\dag}+E_b)a=dd^{\dag}(E_ba)=dd^{\dag}d=d, \label{2.6}\\
&d^{\dag}a=d^{\dag}(dd^{\dag}a)=d^{\dag}d \label{2.7}.
\end{align}

For any $v\in \mathscr{R}$, let
\begin{equation}\label{2.8}
\Phi(v)=v-\frac{1}{2}a^{\dag}avb^{\dag}b+\frac{1}{2}a^{\dag}bv^*a^*(b^{\dag})^*
-\frac{1}{2}a^{\dag}bv^*(b^{\dag}ad^{\dag}a)^*-\frac{1}{2}d^{\dag}avb^{\dag}b.
\end{equation}

In view of (\ref{2.7}) and the definition of $d$, we have
\begin{align*}
a\Phi(v)b^*&=avb^*-\frac{1}{2}avb^*+\frac{1}{2}bv^*a^*bb^{\dag}-
\frac{1}{2}bv^*(bb^{\dag}ad^{\dag}a)^*-\frac{1}{2}ad^{\dag}avb^*\\
&=\frac{1}{2}avb^*+\frac{1}{2}bv^*a^*bb^{\dag}-
\frac{1}{2}bv^*((1-E_b)ad^{\dag}d)^*-\frac{1}{2}ad^{\dag}dvb^*\\
&=\frac{1}{2}avb^*+\frac{1}{2}bv^*a^*bb^{\dag}-\frac{1}{2}bv^*d^{\dag}da^*
+\frac{1}{2}bv^*d^*-\frac{1}{2}ad^{\dag}dvb^*\\
&=\frac{1}{2}avb^*+\frac{1}{2}bv^*a^*bb^{\dag}+\frac{1}{2}bv^*a^*E_b
-H^{(+,*)}(\frac{1}{2}bv^*d^{\dag}da^*)\\
&=H^{(+,*)}(\frac{1}{2}avb^*)-H^{(+,*)}(\frac{1}{2}bv^*d^{\dag}da^*).
\end{align*}
It follows that $\Phi(v)$ is a solution to Eq.(\ref{main1}).

On the other hand, given any solution $x\in \mathscr{R}$ to Eq.
(\ref{main1}), let $v=x$. Then since $d^{\dag}b=0$, we have
\begin{align*}
\Phi(x)&=x-\frac{1}{2}a^{\dag}axb^{\dag}b+\frac{1}{2}a^{\dag}(bx^*a^*)(b^{\dag})^*
-\frac{1}{2}a^{\dag}(bx^*a^*)(b^{\dag}ad^{\dag})^*-\frac{1}{2}d^{\dag}axb^{\dag}b\\
&=x-\frac{1}{2}a^{\dag}axb^{\dag}b+\frac{1}{2}a^{\dag}axb^*(b^{\dag})^*
-\frac{1}{2}a^{\dag}axb^*(b^{\dag}ad^{\dag})^*-\frac{1}{2}d^{\dag}axb^*(b^{\dag})^*\\
&=x-\frac{1}{2}a^{\dag}ax(b^{\dag}ad^{\dag}b)^*-\frac{1}{2}d^{\dag}bx^*a^*(b^{\dag})^*=x.
\end{align*}
We have proved that the general solution to Eq.(\ref{main1}) has a
form $\Phi(v)$ for some $v\in \mathscr{R}$.
\end{proof}

\begin{theorem}
Let $a,b$ be MP-invertible elements in $\mathscr{R}$ such that
$aa^{\dag}b=b$ and $(a^{\dag}bb^{\dag}a)^*=a^{\dag}bb^{\dag}a$, and
$d=E_ba$. Then
\begin{equation}\label{2.4}
x_0=\frac{1}{2}a^{\dag}c(b^{\dag})^*-\frac{1}{2}a^{\dag}bb^{\dag}c(b^{\dag}ad^{\dag})^*
+\frac{1}{2}d^{\dag}c(b^{\dag})^*
\end{equation}
is a solution to the equation
\begin{equation}\label{main2}
axb^*-bx^*a^*=c
\end{equation}
if and only if
\begin{equation}\label{2.2}
c^*=-c \quad and \quad
H^{(-,*)}((aa^{\dag}+dd^{\dag})cbb^{\dag})=2c.
\end{equation}
\end{theorem}

\begin{proof}
If $x_0$ is a solution to equation (\ref{main2}), then obviously
$c^*=-c$ and since $d=E_ba$, in view of (\ref{2.5}) to (\ref{2.7})
we have
\begin{align*}
&aa^{\dag}cbb^{\dag}=aa^{\dag}(ax_0b^*-bx_0^*a^*)bb^{\dag}=ax_0b^*-bx_0^*a^*bb^{\dag},\\
&bb^{\dag}caa^{\dag}=(aa^{\dag}(-c)bb^{\dag})^*=-bx_0^*a^*+bb^{\dag}ax_0b^*,\\
&bb^{\dag}cdd^{\dag}=bb^{\dag}(ax_0b^*-bx_0^*a^*)dd^{\dag}=-bx_0^*a^*dd^{\dag}=-bx_0^*d^*=-bx_0^*a^*E_b,\\
&bb^{\dag}cdd^{\dag}+dd^{\dag}cbb^{\dag}=-bx_0^*a^*E_b+E_bax_0b^*.
\end{align*}

Therefore,
\begin{align*}
&H^{(-,*)}((aa^{\dag}+dd^{\dag})cbb^{\dag})\\
=&aa^{\dag}cbb^{\dag}+bb^{\dag}caa^{\dag}+dd^{\dag}cbb^{\dag}+bb^{\dag}cdd^{\dag}\\
=&ax_0b^*-bx_0^*a^*+(bb^{\dag}+E_b)ax_0b^*-bx_0^*a^*(bb^{\dag}+E_b)\\
=&2(ax_0b^*-bx_0^*a^*)=2c.
\end{align*}

Conversely, suppose that (\ref{2.2}) is satisfied. Let $x_0$ be
defined by (\ref{2.4}). Then as $c^*=-c$ and $d=E_ba$, we have
\begin{align*}
ax_0b^*&=\frac{1}{2}aa^{\dag}cbb^{\dag}-\frac{1}{2}bb^{\dag}c(bb^{\dag}ad^{\dag})^*+\frac{1}{2}ad^{\dag}cbb^{\dag}\\
&=\frac{1}{2}aa^{\dag}cbb^{\dag}-\frac{1}{2}bb^{\dag}c((1-E_b)ad^{\dag})^*+\frac{1}{2}ad^{\dag}cbb^{\dag}\\
&=\frac{1}{2}aa^{\dag}cbb^{\dag}-\frac{1}{2}bb^{\dag}c(ad^{\dag})^*+\frac{1}{2}bb^{\dag}c(dd^{\dag})^*+\frac{1}{2}ad^{\dag}cbb^{\dag}\\
&=\frac{1}{2}H^{(+,*)}(ad^{\dag}cbb^{\dag})+\frac{1}{2}aa^{\dag}cbb^{\dag}+\frac{1}{2}bb^{\dag}cdd^{\dag},
\end{align*}
So
\begin{align*}
H^{(-,*)}(ax_0b^*)&=ax_0b^*-bx_0^*a^*\\
&=\frac{1}{2}aa^{\dag}cbb^{\dag}+\frac{1}{2}bb^{\dag}cdd^{\dag}-(\frac{1}{2}aa^{\dag}cbb^{\dag})^*-(\frac{1}{2}bb^{\dag}cdd^{\dag})^*\\
&=\frac{1}{2}(aa^{\dag}cbb^{\dag}+bb^{\dag}cdd^{\dag}+bb^{\dag}caa^{\dag}+dd^{\dag}cbb^{\dag})\\
&=\frac{1}{2}H^{(-,*)}((aa^{\dag}+dd^{\dag})cbb^{\dag})=c,
\end{align*}
which means that $x_0$ is a solution to Eq.(\ref{main2}).
\end{proof}

Now, we arrive at the most important result of this paper.
\begin{theorem}\label{core}
Let $a,b$ be MP-invertible elements in $\mathscr{R}$ such that
$aa^{\dag}b=b$ and $(a^{\dag}bb^{\dag}a)^*=a^{\dag}bb^{\dag}a$, and
$d=E_ba$. Then Eq.(\ref{main2}) has a solution if and only if
(\ref{2.2}) holds. In which case, the general solution $x$ to
Eq.(\ref{main2}) is of the form $x=x_0+\Phi(v)$, where $v\in
\mathscr{R}$ is arbitrary, and $x_0,\Phi(v)$ are defined by
(\ref{2.4}) and (\ref{2.8}), respectively.
\end{theorem}

In fact, if applying the procedure above to the equation
\begin{equation}\label{main3}
axb^*+bx^*a^*=c,
\end{equation}
we can get a result similar to Theorem \ref{core}, which we will
give in the following theorem, leaving the proof to the read as an
exercise.

\begin{theorem}\label{add}
Let $a,b$ be MP-invertible elements in $\mathscr{R}$ such that
$aa^{\dag}b=b$ and $(a^{\dag}bb^{\dag}a)^*=a^{\dag}bb^{\dag}a$, and
$d=E_ba$. Then Eq.(\ref{main3}) has a solution if and only if
$c^*=c$ and $H^{(+,*)}((aa^{\dag}+dd^{\dag})cbb^{\dag})=2c$. In
which case, the general solution $x$ to Eq.(\ref{main3}) is of the
form $x=x_0'+\Phi'(v)$, where $v\in \mathscr{R}$ is arbitrary, and
$x_0$ and $\Phi(v)$ are defined as the following, respectively:
\begin{align}
&x_0'=\frac{1}{2}a^{\dag}c(b^{\dag})^*-\frac{1}{2}a^{\dag}bb^{\dag}c(b^{\dag}ad^{\dag})^*
+\frac{1}{2}d^{\dag}c(b^{\dag})^*\\
&\Phi'(v)=v-\frac{1}{2}a^{\dag}avb^{\dag}b-\frac{1}{2}a^{\dag}bv^*a^*(b^{\dag})^*
+\frac{1}{2}a^{\dag}bv^*(b^{\dag}ad^{\dag}a)^*-\frac{1}{2}d^{\dag}avb^{\dag}b.
\end{align}
\end{theorem}

If we replace $a$,$b$, and $c$ by $1$, $a$ and $b$, respectively.
From Theorem \ref{add}, we get the following corollary.
\begin{cor}
Let $a\in \mathscr{R}$ be MP-invertible and $b\in \mathscr{R}$. Then
the equation
\begin{equation}\label{2.10}
xa^*+ax^*=b
\end{equation}
has a solution if and only if
\begin{equation}
b^*=b \quad and \quad E_abE_a=0.
\end{equation}
In which case, the general solution $x$ to Eq.(\ref{2.10}) is of the
form
\begin{equation}
x=\frac{1}{2}(1+E_a)(b(a^{\dag})^*-va^{\dag}a)+v-\frac{1}{2}av^*(a^{\dag})^*,
\end{equation}
where $v \in \mathscr{R}$ is arbitrary.
\end{cor}

Similarly, for the equation
\begin{equation}\label{2.13}
a^*x+x^*a=b,
\end{equation}
we also have the following corollary.

\begin{cor}
Let $a\in \mathscr{R}$ be MP-invertible and $b\in \mathscr{R}$. Then
Eq.(\ref{2.13}) has a solution if and only if
\begin{equation}
b^*=b \quad and \quad F_abF_a=0.
\end{equation}
In which case, the general solution $x$ to Eq.(\ref{2.13}) is of the
form
\begin{equation}
x=\frac{1}{2}((a^{\dag})^*b-aa^{\dag}w)(1+F_a)+w-\frac{1}{2}(a^{\dag})^*w^*a,
\end{equation}
where $w \in \mathscr{R}$ is arbitrary.
\end{cor}

\section{The embedding: from rings to rectangular matrices and adjointable operators between Hilbert
$C^*$-modules}\label{embedding}

In this section, we will use the method described in \cite{koliha}
to extend the results for ring $\mathscr{R}$ with involution to the
rectangular matrices over $\mathscr{R}$ and adjointable operators
between Hilbert $C^*$-modules.

The results of the preceding section apply to square matrices of the
same order $n$ over a ring $\mathscr{R}$ as these form a ring
$\mathscr{R}^{n\times n}$ under the usual matrix operations and with
the involution defined as involute transpose. Suppose that $A\in
\mathscr{R}^{m\times n}$, $B \in \mathscr{R}^{p\times m}$ and $C \in
\mathscr{R}^{m\times m}$, consider the equation
\begin{equation}\label{matrix}
AXB^*-BX^*A^*=C \quad \text{for} \quad  X \in \mathscr{R}^{n\times
p}.
\end{equation}
The embedding of Eq.(\ref{matrix}) to a ring is achieved by defining
$a$, $b$ and $c$ in the ring $\mathscr{R}^{k\times k}$, where
$k=m+n+p$, by
\begin{equation}
a=\begin{bmatrix}
0 & A & 0\\
0 & 0 & 0\\
0 & 0 & 0
\end{bmatrix},
\quad b=
\begin{bmatrix}
0 & 0 & B\\
0 & 0 & 0\\
0 & 0 & 0
\end{bmatrix}
\quad \text{and} \quad c=
\begin{bmatrix}
C & 0 & 0\\
0 & 0 & 0\\
0 & 0 & 0
\end{bmatrix}.
\end{equation}

Now consider the equation
\begin{equation}
axb^*-bx^*a^*=c, \quad \text{for} \quad x \in \mathscr{R}^{k\times
k},
\end{equation}
which can also be expressed in the following detailed matrix form,
\begin{multline}\label{detailed}
\begin{bmatrix}
0 & A & 0\\
0 & 0 & 0\\
0 & 0 & 0
\end{bmatrix}
\begin{bmatrix}
X_{11} & X_{12} & X_{13}\\
X_{21} & X_{22} & X_{23}\\
X_{31} & X_{32} & X_{33}
\end{bmatrix}
\begin{bmatrix}
0 & 0 & 0\\
0 & 0 & 0\\
B^* & 0 & 0
\end{bmatrix}
-
\begin{bmatrix}
0 & 0 & B\\
0 & 0 & 0\\
0 & 0 & 0
\end{bmatrix}
\begin{bmatrix}
X_{11}^* & X_{21}^* & X_{31}^*\\
X_{12}^* & X_{22}^* & X_{32}^*\\
X_{13}^* & X_{23}^* & X_{33}^*
\end{bmatrix}
\begin{bmatrix}
0 & 0 & 0\\
A^* & 0 & 0\\
0 & 0 & 0
\end{bmatrix}\\
=
\begin{bmatrix}
C & 0 & 0\\
0 & 0 & 0\\
0 & 0 & 0
\end{bmatrix}.
\end{multline}

Straightforward calculation shows that Eq.(\ref{detailed}) is equal
to the following equation
\begin{equation}
\begin{bmatrix}
AX_{23}B^*-BX_{23}^*A^* & 0 & 0\\
0 & 0 & 0\\
0 & 0 & 0
\end{bmatrix}
=
\begin{bmatrix}
C & 0 & 0\\
0 & 0 & 0\\
0 & 0 & 0
\end{bmatrix}.
\end{equation}
So we have the following lemma.

\begin{lemma}
Let $A\in \mathscr{R}^{m\times n}$, $B \in \mathscr{R}^{p\times m}$
and $C \in \mathscr{R}^{m\times m}$, let $a$, $b$ and $c$ be defined
by (3.2), and let $k=m+n+p$. Then Eq.(\ref{matrix}) has a solution
$X\in \mathscr{R}^{n \times p}$ if and only if Eq.(3.3) has a
solution $x \in \mathscr{R}^{k \times k}$ with $X_{23}=X$. In this
case there is a one-to-one correspondence between the solution $X$
of Eq.(\ref{matrix}) and the solution $x$ of Eq.(3.3) of the form
\begin{equation}
x=
\begin{bmatrix}
0 & 0 & 0\\
0 & 0 & X\\
0 & 0 & 0
\end{bmatrix}.
\end{equation}
\end{lemma}

Note that, if $A$ and $B$ have MP-inverses $A^{\dag}\in
\mathscr{R}^{n\times m}$ and $B^{\dag}\in \mathscr{R}^{m \times p}$,
respectively, then $a$ and $b$ are also MP-invertible with
\begin{equation}
a^{\dag}=
\begin{bmatrix}
0 & 0 & 0\\
A^{\dag} & 0 & 0\\
0 & 0 & 0
\end{bmatrix}
\quad
b^{\dag}=
\begin{bmatrix}
0 & 0 & 0\\
0 & 0 & 0\\
B^{\dag} & 0 & 0
\end{bmatrix}.
\end{equation}
Then it is a work of direct calculation to get the following
theorem, which was obtained by \cite{yuan} in the case of real
matrices.

\begin{theorem}
Let $\mathscr{R}$ be a ring with involution, let $A\in
\mathscr{R}^{m\times n}$, $B \in \mathscr{R}^{p\times m}$ be
MP-invertible and $AA^{\dag}B=B$,
$(A^{\dag}BB^{\dag}A)^*=A^{\dag}BB^{\dag}A$, $D=E_BA$. Then
\begin{equation}
X_0=\frac{1}{2}A^{\dag}C(B^{\dag})^*-\frac{1}{2}A^{\dag}BB^{\dag}C(B^{\dag}AD^{\dag})^*
+\frac{1}{2}D^{\dag}C(B^{\dag})^*
\end{equation}
is a solution to Eq.(3.1)
if and only if
\begin{equation}
C^*=-C \quad \text{and} \quad
H^{(-,*)}((AA^{\dag}+DD^{\dag})CBB^{\dag})=2C.
\end{equation}
In which case, the general solution X to Eq.(3.1) is of the form
$X=X_0+\Phi(V)$, where $V \in \mathscr{R}^{n \times p}$ is
arbitrary, and $\Phi(V)$ is defined by (2.6).
\end{theorem}

Now we turn to the case of adjointable operator between Hilbert
$C^*$-modules. Let $H_1$, $H_2$ and $H_3$ be Hilbert $C^*$-modules,
let $A\in \mathscr{L}(H_3,H_2)$, $B \in \mathscr{L}(H_1,H_2)$ and $C
\in \mathscr{L}(H_2)$ be adjointable operators. The solvability of
the equation
\begin{equation}
AXB^*-BX^*A^*=C \quad \text{for} \quad  X \in \mathscr{L}(H_1,H_3)
\end{equation}
was studied in \cite{xu}. Now we let $H=H_1\oplus H_2\oplus H_3$,
\begin{align}
&\widetilde{A}=
\begin{bmatrix}
0 & 0 & 0\\
0 & 0 & A\\
0 & 0 & 0
\end{bmatrix}\in \mathscr{L}(H),\\
&\widetilde{B}=
\begin{bmatrix}
0 & 0 & 0\\
B & 0 & 0\\
0 & 0 & 0
\end{bmatrix}\in \mathscr{L}(H), \quad \text{and}\\
&\widetilde{C}=
\begin{bmatrix}
0 & 0 & 0\\
0 & C & 0\\
0 & 0 & 0
\end{bmatrix}\in \mathscr{L}(H).
\end{align}
Consider the equation
\begin{equation}
\widetilde{A}\widetilde{X}\widetilde{B}^*-\widetilde{B}\widetilde{X}^*\widetilde{A}^*=\widetilde{C},
\quad \text{for} \widetilde{X}\in \mathscr{L}(H,H).
\end{equation}

Similar to Lemma 3.1, we have
\begin{lemma}
Let $H_1$, $H_2$ and $H_3$ be Hilbert $C^*$-modules, $A\in
\mathscr{L}(H_3,H_2)$, $B \in \mathscr{L}(H_1,H_2)$ and $C \in
\mathscr{L}(H_2)$ be adjointable operators. Then Eq.(3.10) has a
solution $X \in \mathscr{L}(H_1,H_3)$ if and only if Eq.(3.14) has a
solution $\widetilde{X}\in \mathscr{L}(H)$ with
$\widetilde{X}_{31}=X \in \mathscr{L}(H_1,H_3)$. In this case there
is a one-to-one correspondence between the solution $\widetilde{X}$
of Eq.(3.14) and the solution $X$ of Eq.(3.10) of the form
\begin{equation}
\widetilde{X}=
\begin{bmatrix}
0 & 0 & 0\\
0 & 0 & 0\\
X & 0 & 0
\end{bmatrix}.
\end{equation}
\end{lemma}

Note that in Hilbert $C^*$-modules, if $A\in \mathscr{L}(H_3,H_2)$
is MP-invertible, then $R(B)\subseteq R(A)$ is equivalent to
$AA^{\dag}B=B$, and that for any operator $A$, $A$ is MP-invertible
if and only if $R(A)$ is closed. From Theorem 2.4, we reobtain the
following theorem, which was first given in \cite{xu}.
\begin{theorem}
Let $H_1$, $H_2$ and $H_3$ be Hilbert $C^*$-modules, $A\in
\mathscr{L}(H_3,H_2)$, $B \in \mathscr{L}(H_1,H_2)$ and $C \in
\mathscr{L}(H_2)$ be adjointable operators such that $R(B)\subseteq
R(A)$ and $D=E_BA$ such that $R(D)$ is also closed. Then Eq.(3.10)
has a solution if and only if
\begin{equation}
C^*=-C \quad \text{and} \quad
H^{(-,*)}((AA^{\dag}+DD^{\dag})CBB^{\dag})=2C.
\end{equation}
In which case, the general solution X to Eq.(3.10) is of the form
\begin{align*}
X&=X_0+V-\frac{1}{2}A^{\dag}AVB^{\dag}B+\frac{1}{2}A^{\dag}BV^*A^*(B^{\dag})^*\\
&-\frac{1}{2}A^{\dag}BV^*(B^{\dag}AD^{\dag}A)^*-\frac{1}{2}D^{\dag}AVB^{\dag}B,
\end{align*}
where $V \in \mathscr{L}(H_1,H_3)$ is arbitrary, and $X_0$ is a
particular solution to Eq.(3.10) defined by
\begin{equation}
X_0=\frac{1}{2}A^{\dag}C(B^{\dag})^*-\frac{1}{2}A^{\dag}BB^{\dag}C(B^{\dag}AD^{\dag})^*
+\frac{1}{2}D^{\dag}C(B^{\dag})^*.
\end{equation}
\end{theorem}

\bibliographystyle{amsplain}

\end{document}